\documentclass[10pt,a4paper,twoside,openright]{amsart}

\usepackage{amsfonts}
\usepackage{amssymb}
\usepackage{amsmath}
\usepackage{amsthm}
\usepackage{latexsym}
\usepackage{graphicx}
\usepackage{layout}

\newtheorem{theorem}{Theorem}[section]

\newtheorem{proposition}[theorem]{Proposition}
\newtheorem{lemma}[theorem]{Lemma}

\newtheorem{remark}[theorem]{Remark}

\numberwithin{equation}{section}

\renewcommand{\(}{ \left ( }
\renewcommand{\)}{ \right )}
\renewcommand{\[}{ \left [ }
\renewcommand{\]}{ \right ]}

\newcommand{\rr}{ \mathbb{R}}

\newcommand{\R}{\mathbb{R}}

\newcommand{\Z}{\mathbb{Z}}

\newcommand{\cB}{{\mathcal B}}
\newcommand{\cC}{{\mathcal C}}
\newcommand{\cD}{{\mathcal D}}

\newcommand{\cN}{{\mathcal N}}

\newcommand{\eps}{\varepsilon}
\newcommand{\al}{\alpha}
\newcommand{\be}{\beta}
\newcommand{\ga}{\gamma}
\newcommand{\de}{\delta}
\newcommand{\De}{\Delta}
\newcommand{\la}{\lambda}

\newcommand{\Om}{\Omega}

\newcommand{\br}{\bar\rho}

\newcommand{\clos}{\text{\rm clos\,}}

\newcommand{\dist}{\text{\rm dist}}
\newcommand{\spann}{\text{\rm span}}

\newcommand{\id}{\text{\rm id}}
\newcommand{\pa}{\partial}

\newenvironment{altproof}[1]
{\noindent%\addvspace{0.3cm}
{\em Proof of {#1}}.}
{\nopagebreak\mbox{}\hfill $\Box$\par\addvspace{0.5cm}}

\begin{document}

\title[On the profile of  sign changing solutions]{On the profile of
 sign changing solutions of an almost critical problem in the ball}

\author{Thomas Bartsch \& Teresa D'Aprile \& Angela Pistoia}
\address{Thomas Bartsch, Mathematisches Institut, Justus-Liebig-Universit\"at
Giessen, Arndtstr. 2, 34392 Giessen, Germany.}
\email{Thomas.Bartsch@math.uni-giessen.de}
\address{Teresa D'Aprile, Dipartimento di Matematica, Universit\`a di Roma
``Tor Vergata", via della Ricerca Scientifica 1, 00133 Roma, Italy.}
\email{daprile@mat.uniroma2.it }
\address{Angela Pistoia, Dipartimento SBAI,
Universit\`a di Roma ``La Sapienza", via Antonio Scarpa 165, 00161 Roma,
Italy.}
\email{pistoia@dmmm.uniroma1.it }

\thanks{T.~B.\ has been supported by DAAD project 50766047,
Vigoni Project E65E06000080001.}

\thanks{T.~D.\ and A.~P.\ have been supported by   the Italian PRIN Research
Project 2009 \textit{Metodi varia\-zionali e topologici nello studio dei
fenomeni non lineari}.}

\begin{abstract} We study the existence and the profile of sign-changing
solutions to the slightly subcritical problem
$$
-\De u=|u|^{2^*-2-\eps}u \,\hbox{ in } \cB,
 \quad u=0 \,\hbox{ on }\partial  \cB,
$$
where $\cB$ is the unit ball in  $\rr^N$, $N\geq 3$,
$2^*=\frac{2N}{N-2}$ and $\eps>0$ is a small parameter. Using a
Lyapunov-Schmidt reduction we discover two new non-radial solutions
having 3 \textit{bubbles} with different nodal structures. An interesting
feature is that the solutions are obtained as a local minimum and a local
saddle point of a reduced function, hence they do not have a global
min-max description.
\end{abstract}
%\vspace{.2cm}

\maketitle

{\small


}

\section{Introduction and  main result}\label{intro}

The paper is concerned with the slightly subcritical elliptic problem
\begin{equation}\label{eq1}
-\De u=|u|^{2^*-2-\eps}u\ \hbox{in}\ \Om,
 \quad u=0\ \hbox{on}\ \partial \Om
\end{equation}
where $\Omega$ is a smooth and bounded domain in $\rr^N$, $N\geq 3$,
$\eps>0$ is a small parameter. Here $2^*$ denotes the critical exponent
in the Sobolev embeddings, i.e. $2^*=\frac{2N}{N-2}$. This problem has
received a lot of attention, in particular with respect to investigating
the lack of compactness of the critical problem where $\eps = 0$. Whereas
most papers deal with positive solutions and their blow-up behavior as
$\eps \to 0$ we deal with sign-changing solutions.

In \cite{poho} Poho\u{z}aev proved that the problem \eqref{eq1} does not
admit a nontrivial solution if $\Omega$ is star-shaped and $\eps\leq0$.
On the other hand problem \eqref{eq1} has a positive solution if $\eps\leq0$
and $\Omega$ is an annulus, see Kazdan and Warner \cite{kaz}. In \cite{baco}
Bahri and Coron found a positive solution to \eqref{eq1} with $\eps = 0$
provided that the domain $\Omega$ has a nontrivial topology. The
slightly supercritical case $\eps<0$ was studied in
\cite{delfemu1,delfemu2,delfemu3,pire} where the authors proved solvability
of \eqref{eq1} for $\eps<0$ sufficiently small in a domain with one or more
small holes and found positive solutions which blow up at 2 or more points
of the domain as $\eps$ goes to zero, i.~e.\ appropriately scaled solutions
$u_\eps$ converge, as $\eps\to 0^+$, towards a sum of delta distributions.

In the  subcritical case $\eps >0$, the Rellich-Kondrachov compact embedding
theorem ensures the exi\-stence of at least one positive solution and of
infinitely many sign changing solutions. In \cite{brepe,fluwe,han,rey1,rey3}
it was proved that, as $\eps\to 0^+$, the least energy positive solution
blows up and concentrates at a point $\xi$ which is a critical point of the
Robin's function of $\Omega$. Successively, in \cite{balirey,rey2} it was
studied the existence of positive solutions of \eqref{eq1} with $k\ge2$
blow-up points. In particular, in \cite{grotak} it was proved that in a
convex domain problem \eqref{eq1} does not admit any positive solution
blowing up at $k\ge2$ points.

Concerning sign-changing solutions we mention the papers
\cite{ba,Ba-Wang:1996,bawe} where the authors provide existence and
multiplicity of sign-changing solutions for more general problems than
\eqref{eq1}. However, these papers are not concerned with the nodal structure
of the solutions. The question whether the nodal surface
$$
\cN(u) := \clos\{x \in \Om: u(x) = 0\}
$$
of a solution $u$ of \eqref{eq1} intersects the boundary $\pa\Om$ has been
investigated in \cite{Aftalion-Pacella:2004}. The number and shape of nodal
domains is also important when one investigates competing species or phase
separation problems in systems like
\begin{equation}\label{eq:system}
\left\{
\begin{aligned}
-\De v &= |v|^{2^*-2-\epsilon}v + \be G_v(v,w),\\
-\De w &= |w|^{2^*-2-\epsilon}w + \be G_w(v,w),\\
v,w &\in H^1_0(\Om).
\end{aligned}
\right.
\end{equation}
It has been proved in \cite{conti-terracini-verzini:2002}, under appropriate
conditions on $G$ modeling the competition, that as $\be\to-\infty$ positive
solutions $(v_\be,w_\be)$ of \eqref{eq:system} converge towards $(v,w)$, such
that $v\cdot w=0$ and $u=v-w$ solves \eqref{eq1}. Thus the nodal domains of
$u$ correspond to the domains where the competing species $v,w$ live.

In \cite{bamipi} a solution with exactly one positive and one negative
blow-up point is constructed for problem \eqref{eq1} if $\eps>0$ is
sufficiently small. The location of the two blow-up points is also
characterized and depends on the geometry of the domain. Moreover, the authors
proved that when $\Omega$ is a ball, for any integer $k$ there exists a
solution with $k$ positive peaks and $k$ negative peaks which are located at
the vertices of a regular polygon. In particular, the nodal regions of these solutions always intersect the boundary (see Fig.~\ref{fig1}).

\begin{figure}
\includegraphics[width=\textwidth]{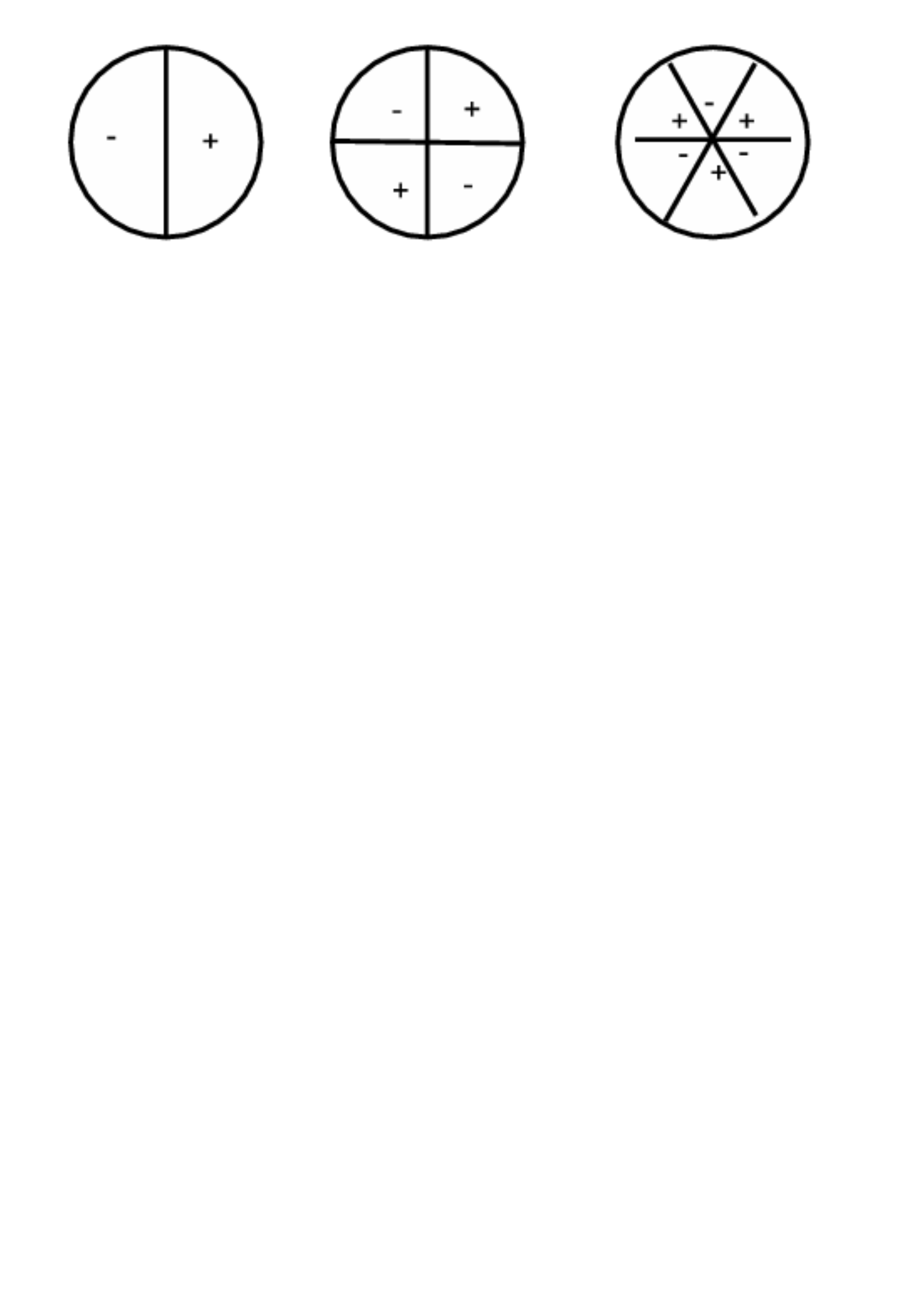}.
\vspace{-15truecm}
\caption{The nodal structure of the solutions with 2 peaks, 4 peaks, 6 peaks
found in \cite{bamipi}.}
\label{fig1}
\end{figure}

All the previous results deal with solutions with many simple blow up points.
The presence of sign-changing solutions with a multiple blow-up point is
observed in \cite{mupi1} and  \cite{piwe} for problem \eqref{eq1}. In these
papers the authors obtain solutions which have the shape of towers of
alternating-sign bubbles, i.~e.\ they are constructed as superpositions of
positive bubbles and negative bubbles blowing-up at the same point with a
different concentration rate. As a consequence, the nodal regions of these
solutions shrink to the blow up point as $\eps$ goes to zero. We also
mention the paper \cite{bep}, where the authors  study the blow-up of the
low energy sign-changing solutions of problem \eqref{eq1} and they classify
these solutions according to the concentration speeds of the positive and
negative part. In particular, they obtain some qualitative results, such as
symmetry or location of the concentration points when the domain is a ball.

Recently, the authors investigated the existence of solutions in a convex
and symmetric domain blowing up positively at $k$ points and negatively at
$l$ different points which are aligned with alternating sign along the
symmetry axis as $\eps\to0^+$. The case $k = l = 1$ has been settled in
\cite{bamipi}, the case $k = l = 2$ in \cite{badapi};
(see Fig.\ \ref{fig2}).
\begin{figure}
\includegraphics[width=\textwidth]{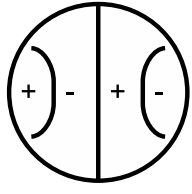}.
\vspace{-13truecm}
\caption{The nodal structure of the solution with 4 peaks from
\cite{badapi}. }
\label{fig2}
\end{figure}
Other cases, where $k = l \ge3$ or $k = l-1$, remained open. In fact, this
type of solutions seems to be very hard to find.

In the present paper we are able to prove the existence of a sign changing
solution which blows up positively at one point and negatively at two points
 when $\Omega=\cB$ is the open unit ball in $\R^N$. In order to formulate
the result we introduce the functions
\begin{equation}\label{soliton}
U_{\delta,  \xi}(x) = \alpha_N
 \bigg(\frac{\delta}{\delta^{2}+|x-\xi|^2}\bigg)^{(N-2)/2},
 \quad\alpha_N=(N(N-2))^{(N-2)/4},
\end{equation}
where $\delta>0$ and $\xi\in\rr^N$. These are actually all positive solutions
of the  limiting equation
$$
-\De U=U^{2^*-1}\hbox{ in }\rr^N
$$
and constitute the extremals for the Sobolev's critical embedding (see
\cite{au,cagispru,tal}). Let $P:\cD^{1,2}(\R^N)\to H^1_0(\cB)$ denote the
orthogonal projection with respect to the scalar product
$(u,v)=\langle \nabla u,\nabla v\rangle_{L^2}$.

\begin{theorem}\label{main}
There exist $\eps_0>0$ and for $0<\eps<\eps_0$ solutions
$\pm u_{1,\eps},\pm u_{2,\eps}\in H^1_0(\cB)$ of
\begin{equation}\label{1}
-\De u=|u|^{2^*-2-\epsilon}u\ \hbox{in}\ \cB,
\quad u=0\ \hbox{on}\ \partial \cB,
\end{equation}
with the following properties.

a) The solutions are nonradial and even with respect to $x_1,\dots,x_N$,
and their limiting behavior as $\eps \to 0$ is of the form
$$
u_{i,\eps} = PU_{\ga_{i,\eps},0} - PU_{\de_{i,\eps},\xi_{i,\eps}}
              - PU_{\de_{i,\eps},-\xi_{i,\eps}} + O(\eps)
\quad \text{in } H^1_0(\cB)
$$
where $\ga_{i,\eps},\de_{i,\eps} > 0$ and $\xi_{i,\eps} \in \cB\setminus\{0\}$.

b) $\ga_{i,\eps}/\eps$ and $\de_{i,\eps}/\eps$ are bounded away from $0$
and $\infty$, hence the solution $u_{i,\eps}$ has one positive blow-up
point at $0$ and two negative blow-up points at $\pm\xi_{i,\eps}$.

c) The blow-up points $\xi_{i,\eps}$ are bounded away from $0$ and $\pa\cB$,
and they satisfy $0<|\xi_{1,\eps}|<|\xi_{2,\eps}|<1$.

d) The exterior normal derivative of $u_{1,\eps}$ changes sign on $\pa\cB$,
whereas the exterior normal derivative of $u_{2,\eps}$ is strictly positive
on $\pa\cB$.
\end{theorem}

\begin{figure}
\includegraphics[width=\textwidth]{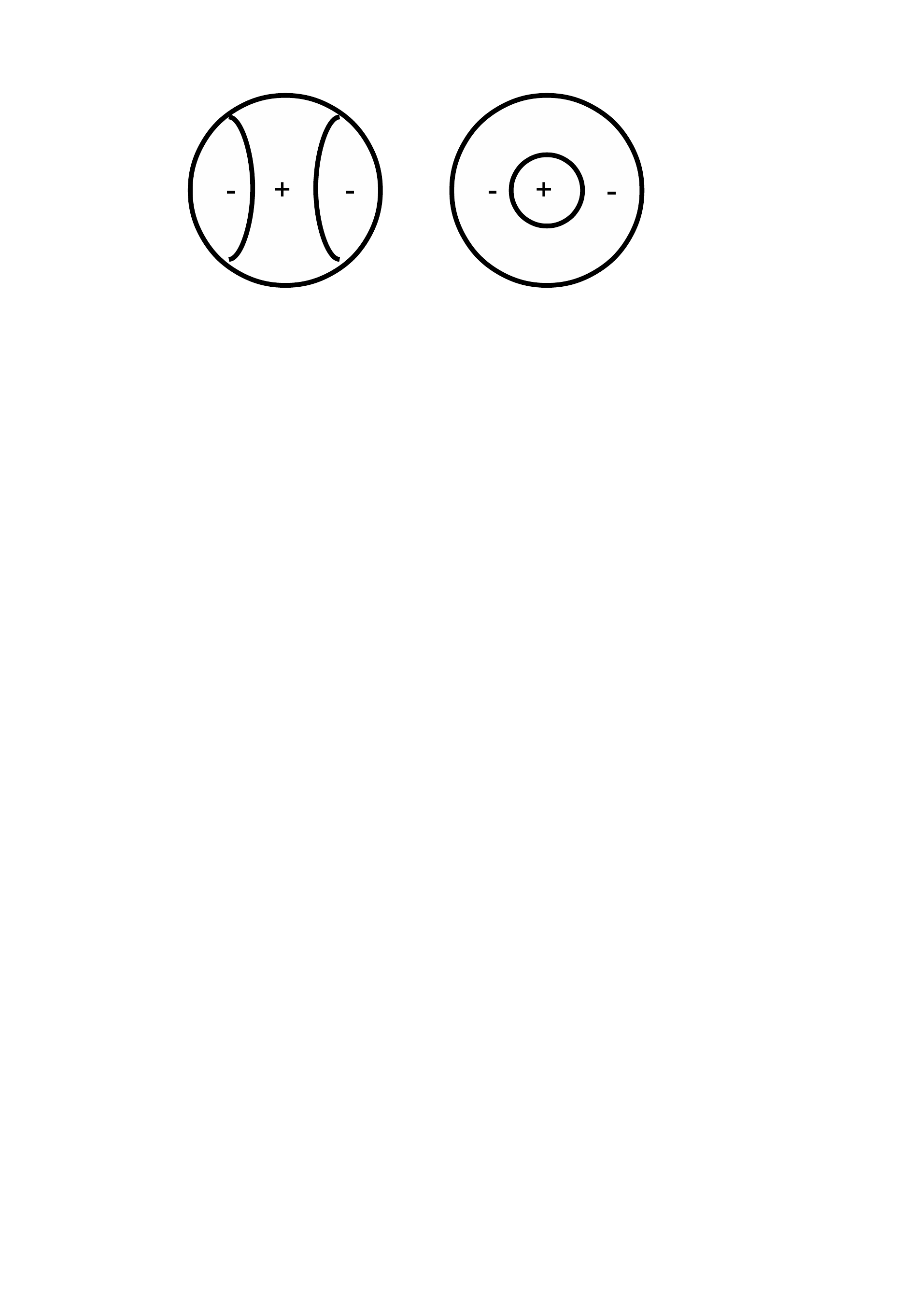}.
\vspace{-14truecm}
\caption{The nodal structure of the two solutions with 3 peaks from
Theorem~\ref{main}.}
\label{fig3}
\end{figure}

\begin{remark}\label{rem:main}
a) As a consequence of Theorem~\ref{main} d) the nodal surface
$\cN(u_{1,\eps})$ of $u_{1,\eps}$ intersects the boundary, and the nodal
surface $\cN(u_{2,\eps})$ of $u_{2,\eps}$ does not intersect the boundary
(see Fig.~\eqref{fig3}).

b) Since the problem is radially invariant, for any $A\in O(N)$ the function
$u\circ A$ solves \eqref{1} if $u$ does. Hence the solutions and the blow-up
points $\pm\xi_{i,\eps}$ are determined up to rotations.

c) The proof shows that the solutions $(\eps,u_{1,\eps}),(\eps,u_{2,\eps})$
lie on connected sets $\cC_1,\cC_2\subset(0,\eps_0)\times H^1_0(\cB)$,
respectively. We expect that $u_{1,\eps},u_{2,\eps}$ are the only solutions
of \eqref{1} having the form as described in Theorem~\ref{main}, and
consequently, that the maps
$$
u_i:(0,\eps_0)\to H^1_0(\cB),\quad \eps\mapsto u_{i,\eps},\qquad (i=1,2)
$$
are continuous; see Remarks~\ref{rem:exp1}, \ref{rem:exis}.

d) The energy of $u_{1,\eps}$ is larger than the energy of $u_{2,\eps}$.
The Morse indices $m(u_{i,\eps})$ of the solutions differ by $1$, i.~e.\
$m(u_{1,\eps}) = m(u_{2,\eps}) + 1$. It follows from
\cite[Theorem~1.2]{Aftalion-Pacella:2004} that $m(u_{2,\eps}) \ge N+1$.
From a Morse theoretic point of view the topology generated by
$u_{2,\eps}$, the solution with smaller energy, is canceled by
$u_{1,\eps}$.
\end{remark}

The proof of Theorem~\ref{main} relies on a well known Ljapunov-Schmidt
procedure. This will be recalled in Section~\ref{setting} where we reduce
the problem to a finite-dimensional one. In Section~\ref{existence} we
prove the existence of two critical points (one local minimum point and
one local saddle point) of the reduced energy. These two critical points
generate the two solutions of problem~\eqref{1}. Finally, in
Section~\ref{profile} we show that the normal
derivative of the solution generated by the minimum point changes sign on
the boundary, while the normal derivative of the solution generated by
the saddle point does not change sign on the boundary of the ball.

It seems very hard to generalize this result to more general domains,
except for small perturbations of the ball. Even the case of an an ellipsoid
seems to be very difficult. In fact, in the case of the ball the Green's
function of the Laplace operator is known and it allows to find the two
local critical points of the reduced energy essentially by direct
computations. In general, the Green's function is not explicitely known and
so it becomes very difficult to find critical points of local type, i.~e.\
local minima or local saddle points. A global min-max scheme to obtain
these solutions does not seem to be possible according to
Remark~\ref{rem:main}~d).

\section{Setting of the problem }\label{setting}

The proof of Theorem \ref{main} is based on a \textit{finite dimensional
reduction} procedure. We sketch the procedure here and refer to
\cite{bamipi} for details. To begin with we introduce scaled versions of the
functions from \eqref{soliton}:
$$
U_{\la,\rho}^\eps = U_{\la^2\eps,(\rho,0)},\quad \la>0,\ \rho\in(0,1);
$$
here $(\rho,0)\in\cB\subset\R\times\R^{n-1}$. We look for symmetric solutions
of \eqref{1} of the form
\begin{equation}\label{2}
u_\eps:=PU_{\la,0}^\eps-PU_{\mu,\rho}^\eps-PU_{\mu,-\rho}^\eps+\phi
\end{equation}
with $\la,\mu>0$, $\rho\in(0,1)$, and $\phi=O(\eps)$. Moreover, $\phi$
belongs to a suitable space defined below. It is useful to recall that
$$
PU_{\de,\xi}(x) = U_{\de,\xi}(x) - \ga_N\delta^{(N-2)/2}H(x,\xi)
   +O\left({\delta^{(N+2)/2}\big/(\dist(\xi,\pa\cB))^N}\right)
$$
where $\ga_N>0$ is a constant and the function $H$ is the regular
part of the the Green's function $G$ of the Laplace operator in the ball,
i.~e.\
$$
G(x,y)=\frac{1}{|x-y|^{N-2}}-H(x,y),\quad\
H(x,y)=\frac{1}{(|x|^2|y|^2+1-2(x,y))^{(N-2)/2}}\ .
$$
By the principle of symmetric criticality critical points of the energy
functional
$$
J_\eps: H^1_0(\cB)\to\R, \quad
J_\eps(u)
 = \frac12\int_{\cB}|\nabla u|^2 dx
    - \frac{1}{2^*-\eps}\int_{\cB}|u|^{2^*-\eps}dx,
%\quad u\in H^1_0(\cB),
$$
constrained to the subspace
$$
H_e := \{u \in H^1_0(\cB): u \text{ is even in } x_1,\dots,x_N\}
    \subset H^1_0(\cB),
$$
are solutions to problem \eqref{1}. In order to define the space for
$\phi$ we set
$$
K^\eps_{\la,\mu,\rho }
 :=\spann\left\{P\left(\frac{\pa}{\pa\la} U_{\la,0}^\eps\right),\
  P\left(\frac{\pa}{\pa\mu} U_{\mu,\rho}^\eps\right),\
  P\left(\frac{\pa}{\pa\rho} U_{\la,\rho}^\eps\) \right\}\subset H_e,
$$
and
$$
\left(K^\eps_{\la,\mu,\rho}\right)^\perp
 := \left\{\phi\in H_e:(\phi,\psi)=0\
     \hbox{for any }\psi\in K^\eps_{\la,\mu,\rho} \right\}
 \subset H_e.
$$
Here $(u,v)=\langle \nabla u,\nabla v\rangle_{L^2}$ will be used as inner
product on the Hilbert space $H_e$. We write
$\|u\|=\|\nabla u\|_{L^2}$ for the associated norm on $H_e$.

We first solve an intermediate problem for $\phi$ (see \cite{bamipi}).
Define
$$
V^\eps_{\la,\mu,\rho}:=PU_{\la,0}^\eps-PU_{\mu,\rho}^\eps-PU_{\mu,-\rho}^\eps,
$$
with $(\la,\mu,\rho) \in \cD := (0,\infty)\times(0,\infty)\times (0,1)$,
and, for any $\eta>0$ small,
$$
\cD_\eta:=(\eta,\eta^{-1})\times(\eta,\eta^{-1})\times (\eta,1-\eta)\subset D.
$$

\begin{lemma}\label{reg}
For $\eta>0$ small there exists $\eps_1>0$ and a constant $C>0$ such that for
each $\eps\in (0,\eps_1)$ and each $(\la,\mu,\rho)\in \mathcal{D}_\eta$ there
exists a unique
$\phi=\phi^\eps_{\la,\mu,\rho} \in \(K^\eps_{\la,\mu,\rho}\)^\perp$
satisfying
$$
\De(V^\eps_{\la,\mu,\rho}+\phi)
  + |V^\eps_{\la,\mu,\rho}+\phi|^{2^*-2-\eps}(V^\eps_{\la,\mu,\rho}+\phi)
 \in K^\eps_{ \la,\mu,\rho}
$$
and $ \|\phi\| < C\eps$. Moreover, the map
$D_\eta\ni(\la,\mu,\rho)\mapsto \phi^\eps_{\la,\mu,\rho}\in H_e$ is of
class $C^1$.
\end{lemma}

Now we introduce the reduced energy functional
$$
\widetilde J_\eps: \cD_\eta \to\R, \quad
\widetilde J_\eps(\la,\mu,\rho)
 := J_\eps\(V^\eps_{\la,\mu,\rho} + \phi^\eps_{\la,\mu,\rho}\),
$$
where $\phi^\eps_{\la,\mu,\rho}$ has been constructed in Lemma \ref{reg}.
The next result (see \cite{balirey}) reduces the original problem \eqref{1}
to a finite dimensional one.

\begin{proposition}\label{relation}
The point  $(\la,\mu,\rho) \in \cD_\eta$ is a critical point
of $\widetilde J_\eps$ if, and only if, the corresponding function
$u_\eps = V^\eps_{\la,\mu,\rho} + \phi^\eps_{\la,\mu,\rho}$
is a solution of \eqref{1}.
\end{proposition}

Now we expand the reduced energy.

\begin{proposition}\label{exp1}
We have
\begin{equation}\label{lonel}
\widetilde J_\eps (\la,\mu,\rho)
 = {c_1}_N+{c_2}_N\eps\log\eps+{c_3}_N\eps+{c_4}_N\eps F(\la,\mu,\rho)+o(\eps)
\end{equation}
${\mathcal C}^1$-uniformly on compact sets of $\cD$. Here the ${c_i}_N$'s are positive
constants which depend only on $N$. Moreover
\begin{equation}\label{4}
F(\la,\mu,\rho) :=  a\la^2 + 2\mu^2\alpha(\rho) + 4\la\mu\beta(\rho)
 - c_n\ln\la - 2c_n\ln\mu,
\end{equation}
with $c_N$ a positive constant depending only on $N$, and $a:=H(0,0)=1$,
$$
\begin{aligned}
\alpha(\rho) &:= H((\rho,0),(\rho,0)) - G((\rho,0),(-\rho,0))\\
 &= \frac{1}{(1-\rho^2)^{n-2}} - \frac{1}{(2\rho)^{N-2}}
    + \frac{1}{(1+\rho^2)^{N-2}},
\end{aligned}
$$
$$
\beta(\rho) := G((\rho,0),(0,0)) = \frac{1}{\rho^{N-2}}-1.
$$
\end{proposition}

\begin{proof}
The proof proceeds as in \cite{delfemu2,delfemu3}; see also
\cite[Proposition~3.1]{bamipi}.
\end{proof}

\begin{remark}\label{rem:exp1}
If $N<6$ then the original functional is of class $\cC^3$ for $\eps$
small, hence the map
$D_\eta\ni(\la,\mu,\rho)\mapsto \phi^\eps_{\la,\mu,\rho}\in H_e$
from Lemma~\ref{reg} is of class $\cC^2$. Then the reduced functional
$\widetilde J_\eps$ is of class $\cC^2$, and it can be proved as in
\cite[Proposition~2.3]{Glangetas:1993} that the expansion in
Proposition~\ref{exp1} is $\cC^2$-uniformly on compact sets of $\cD$.
\end{remark}

Our main result will follow from the following proposition applied to
$\nabla\widetilde J_\eps$.

\begin{proposition}\label{thm:cont}
Let $U\subset\R^m$ be open and bounded, and consider a one parameter
family of maps $h_\eps: \overline{U} \to \R^m$ of the form
$h_\eps(x) = \eps f(x) + g_\eps(x)$
with $f,g_\eps: \overline{U} \to \R^m$ of class $\cC^k$, $k \ge 0$. Suppose
the map $(0,\eps_1) \to \cC^k(\overline{U},\R^m)$,
$\eps \mapsto g_\eps$, is $\cC^k$ and satisfies
$\|g_\eps\|_{\cC^k} = o(\eps)$ as $\eps \to 0$.

{\rm a)} If the Brouwer degree $\deg(f,U,0) \ne 0$ is
well defined and nontrivial then there exists $\eps_0 > 0$ and a connected
subset $\cC\subset(0,\eps_0) \times U$ with the following properties:
\begin{itemize}
\item[\rm (i)] $\cC$ covers the interval $(0,\eps_0)$, i.~e.\ for every
$0 < \eps < \eps_0$ there exists $x_\eps\in U$ with $(\eps,x_\eps) \in \cC$.
\item[\rm (ii)] If $(\eps,x)\in\cC$ then $h_\eps(x) = 0$.
\item[\rm (iii)] Given a sequence $(\eps_n,x_n) \in \cC$ with $\eps_n \to 0$,
$x_n \to x_0$, then $f(x_0) = 0$.
\end{itemize}

{\rm b)} If $k = 1$ and if $x_0 \in U$ is a nondegenerate zero of $f$ then
there exists $\eps_0 > 0$ and a $\cC^1$-map $(0,\eps_0) \to U$,
$\eps \mapsto x_\eps$, such that $h_\eps(x_\eps) = 0$ for all
$\eps \in (0,\eps_0)$. Moreover, $x_\eps \to x_0$ as $\eps \to 0$.
\end{proposition}

\begin{proof}
a) The maps $\eps^{-1} h_\eps = f + \eps^{-1} g_\eps$ converge
uniformly towards $f$ as $\eps \to 0$, hence
$\deg(\eps^{-1} h_\eps,U,0) = \deg(f,U,0)\ne0$ is well defined and
nontrivial for $\eps > 0$ small. The existence and the properties of
$\cC$ follow by standard arguments.

b) Existence and uniqueness of $x_\eps$ follow from the contraction mapping
principle applied to
$$
f_\eps := \id - \eps^{-1}Df(x_0)^{-1}\circ h_\eps
 = \id - Df(x_0)^{-1}\circ f - \eps^{-1}Df(x_0)^{-1}\circ g_\eps.
$$
For $\de,\eps > 0$ small $f_\eps: B_\de(x_0) \to B_\de(x_0)$ is well defined
and a contraction. Differentiability is a consequence of the implicit function
theorem.
\end{proof}

Finally, the propositions \ref{exp1}, \ref{thm:cont}, and
\cite[Lemma 4.5]{bamipi} imply the following existence theorem. Observe that
$F$ is analytic, hence a critical point $(\la,\mu,\rho)$ of $F$ is
automatically isolated, and its degree,
$\deg(F,(\la,\mu,\rho)) := \deg(\nabla F,U_\de(\la,\mu,\rho),0) \in \Z$
for $\de > 0$ small, is well defined.

\begin{theorem}\label{th-ex}
Let $(\tilde\la,\tilde\mu,\tilde\rho)$ be a critical point of $F$ with
nontrivial degree.

a) There exist $\tilde\eps > 0$ and a connected subset
$\widetilde\cC\subset(0,\tilde\eps) \times \cD$ with the following properties:
\begin{itemize}
\item[(i)] $\widetilde\cC$ covers the interval $(0,\tilde\eps)$.
\item[(ii)] If $(\eps,\la,\mu,\rho) \in \widetilde\cC$ then
$\nabla\widetilde J_\eps(\la,\mu,\rho) = 0$.
\item[(iii)] Given a sequence $(\eps_n,\la_n,\mu_n,\rho_n)\in\widetilde\cC$
with $\eps_n \to 0$, then
$(\la_n,\mu_n,\rho_n) \to (\tilde\la,\tilde\mu,\tilde\rho)$.
\item[(iv)] Setting
$$
\cC := \{(\eps,u): u=V^\eps_{\la,\mu,\rho} + \phi^\eps_{\la,\mu,\rho},\
          					(\eps,\la,\mu,\rho)\in\widetilde\cC\}
		\subset (0,\tilde\eps)\times H^1_0(\cB)
$$
then any family $(\eps,u_\eps)\in\cC$ converges in
$C^1_{loc}\(\overline \Omega\setminus\{(0,0),(\tilde\rho,0),(-\tilde\rho,0)\}\)$
\begin{equation}\label{profi}
\frac{1}{\sqrt\eps} u_\eps(x) \to
  \alpha_N\big(\tilde\la G(x,0) - \tilde\mu G\(x,(\tilde\rho,0)\)
  -\tilde\mu G\(x,(-\tilde\rho,0) \right)
\quad \hbox{as } \eps\to 0.
\end{equation}
\end{itemize}

b) If $(\tilde\la,\tilde\mu,\tilde\rho)$ is a nondegenerate critical point
of $F$ then the set $\widetilde\cC$ from a) is the graph of a $\cC^1$-map
$(0,\tilde\eps) \to \cD$. Correspondingly, the set $\cC$ is the graph of a
$\cC^1$-map $(0,\tilde\eps) \to H^1_0(\cB)$.
\end{theorem}

%We say that the solution $u_\eps$ is generated by the critical point
%$(\la^*,\mu^*,\rho^*).$

%
\section{The existence of two solutions}\label{existence}

In this section we will prove Theorem~\ref{main} a), b), c).

\begin{lemma}\label{exis}
The function $F$ from \eqref{4} has two isolated critical points
$(\la_1,\mu_1,\rho_1)$ and $(\la_2,\mu_2,\rho_2)$. The first one is a local
saddle point with Morse index 1, hence it has degree
$\deg(F,(\la_1,\mu_1,\rho_1)) = -1$. The second is a strict local minimum,
hence $\deg(F,(\la_1,\mu_1,\rho_1)) = 1$. Moreover,
$ \rho_1 < \frac12 < \rho_2$.
\end{lemma}

Postponing the proof of this lemma we first deduce the\\

\begin{altproof}{Theorem~\ref{main} a), b), c)}
Let $u_{i,\eps}$ be the solution of \eqref{1} corresponding to the critical
point $(\la_i,\mu_i,\rho_i)$ of $F$ from Lemma~\ref{exis}. Then the blow-up
property of Theorem~\ref{main}~a) is satisfied by construction (see
Theorem~\ref{th-ex}) with
$$
\xi_{i,\eps} \to (\rho_i,0),\
\frac{\ga_{i,\eps}}{\eps} \to \sqrt{\la_i},\
\frac{\de_{i,\eps}}{\eps} \to \sqrt{\mu_i},
$$
as $\eps \to 0$. Properties b) and c) of Theorem~\ref{main} follow
immediately.
\end{altproof}

Theorem~\ref{main}~d) will be proved in the next section.

\begin{remark}\label{rem:exis}
We conjecture that $F$ has precisely two critical points. Then we obtain
continuous curves
$$
(0,\eps_0)\to H_e \subset H^1_0(\cB), \quad\eps\mapsto u_{i,\eps},
\qquad (i=1,2)
$$
of sign changing solutions $u_{i,\eps}$ to problem \eqref{1} as stated in
Remark~\ref{rem:main}. If $N<6$ we checked numerically that the two critical
points from Lemma~\ref{exis} are nondegenerate, hence Theorem~\ref{th-ex}
implies that the two curves are of class $\cC^1$.
\end{remark}
	
\begin{altproof}{Lemma~\ref{exis}}
First of all, it is useful to point out that, since $\alpha'>0$, 
$\al(\rho)\to -\infty$ as $\rho\to 0$, $\alpha(\frac{1}{2})>0$, there exists 
$\rho_0\in(0,\frac12)$ such that
$$\alpha(\rho_0)=0\ \hbox{and}\ \alpha(\rho)>0\ \hbox{for any}\ \rho \in(\rho_0,1).$$
We have
$$
\partial_\la F(\la,\mu,\rho) = 2\la + 4\mu\beta(\rho) - \frac{c_N}{\la},\quad
\partial_\mu F(\la,\mu,\rho)
 = 4\mu\alpha(\rho) + 4\la\beta(\rho) - \frac{2c_N}{\mu}.
$$
Then for any $\rho\in(\rho_0,1)$ there exist  unique $\la(\rho)$ and $\mu(\rho)$
such that
$$
\nabla_{\la,\mu}F(\la(\rho),\mu(\rho),\rho) = 0.
$$
More precisely,
\begin{equation}\label{7}
\mu(\rho) = \sqrt{\frac{c_N}{2\alpha(\rho)+2\Lambda(\rho)\beta(\rho)}},\qquad
\la(\rho)=\Lambda(\rho)\mu(\rho)
\end{equation}
where
\begin{equation}\label{6}
\Lambda(\rho) := \frac{\sqrt{\beta^2(\rho)+4 \alpha(\rho)}-\beta (\rho)}{2} > 0.
\end{equation}
We remark that the Hesse matrix
$ D^2_{\la,\mu}F(\la(\rho),\mu(\rho),\rho)$ is positively definite and in
particular non-degenerate. In fact, since
$\nabla_{\la,\mu}F(\la(\rho),\mu(\rho),\rho) = 0$, an easy computation shows
that
\begin{align*}
D^2 _{\la,\mu}F(\la(\rho),\mu(\rho),\rho)
 &= \(\begin{aligned}
     & 2 + \frac{c_N}{\(\la(\rho)\)^2} & 4\beta(\rho)\\
     & 4\beta(\rho) & 4\alpha(\rho)+ \frac{2c_N}{\(\mu(\rho)\)^2} \\
		\end{aligned}\)\\
 &= 16 \(\begin{aligned}
           & 1+\frac{\beta(\rho)}{\Lambda(\rho)} & \beta(\rho)\\
		       & \beta(\rho) & 2 \alpha(\rho)+ \Lambda(\rho)\beta(\rho) \\
		     \end{aligned}\).
\end{align*}
Then
$\mathrm{tr}D^2 _{\la,\mu}F(\la(\rho),\mu(\rho),\rho)>0 $
and
$$
\mathrm{det}D^2_{\la,\mu}F(\la(\rho),\mu(\rho),\rho)
 = 16\(2\alpha(\rho) + \Lambda(\rho)\beta(\rho)
    + \frac{2\alpha(\rho)\beta(\rho)}{\Lambda(\rho)}\) > 0.
$$

Now, let us consider the reduced function
\begin{equation}\label{def:f}
f(\rho) := F(\la(\rho),\mu(\rho),\rho)
 = \frac32 c_N-c_N\log\[\la(\rho)\mu^2(\rho)\],\;\; \rho\in (\rho_0,1).
\end{equation}
Since $f:(\rho_0,1) \to \R$ is analytic, critical points are either strict
local maxima or minima. If $\rho_1$ is a local maximum of $f$ then
$(\la(\rho_1),\mu(\rho_1),\rho_1)$ is a critical point of $F$ with Morse
index $1$ and degree $-1$. If $\rho_2$ is a local minimum of $f$ then
$(\la(\rho_2),\mu(\rho_2),\rho_2)$ is a local minimum of $F$ with
degree $+1$.

Observe that
\begin{equation}\label{limiti}
\lim\limits_{\rho\to\rho_0^+}f(\rho) = -\infty \quad\hbox{and}\quad
\lim\limits_{\rho\to1^-}f(\rho) = +\infty.
\end{equation}
In fact, since $\alpha(\rho_0) = 0$ we get as $\rho\to\rho_0^+$
$$
\Lambda(\rho) \sim \alpha(\rho),\ \ \mu(\rho) \sim \(\alpha(\rho)\)^{-1/2},\ \
\la(\rho) \sim \(\alpha(\rho)\)^{1/2} $$
which imply
$$
\la(\rho)\mu^2(\rho)\to+\infty \quad\hbox{and}\quad
f(\rho) \to -\infty \quad\text{as } \rho \to \rho_0^+.
$$
Moreover, as $\rho\to1^-$ we have
$$
\alpha(\rho) \sim \frac{1}{(1-\rho)^{N-2}},\ \ \beta(\rho) \sim -\rho,\ \
\Lambda(\rho) \sim \sqrt{\alpha(\rho)} \sim \frac{1}{(1-\rho)^{\frac{N-2}{2}}},
$$
which imply
$$
\mu(\rho) \sim (1-\rho)^{\frac{N-2}{2}},\ \ \la(\rho) \sim 1
\quad\hbox{and}\quad f(\rho) \to +\infty \quad\text{as } \rho \to 1^-.
$$

We claim that
\begin{equation}\label{one}
f'\Big(\frac12\Big) < 0 \quad\text{ for all }N\geq 3.
\end{equation}
Then it follows that $f$ has a local maximum point $\rho_1 < \frac12$ and a
local minimum point $\rho_2 > \frac12$.

An easy computation shows that
\begin{align*}
f'(\rho)
 &= \pa_\la F(\la(\rho),\mu(\rho),\rho)\la'(\rho)
    + \pa_\mu F(\la(\rho),\mu(\rho),\rho)\mu'(\rho)
		+ \pa_\rho F(\la(\rho),\mu(\rho),\rho)\\
 &= \pa_\rho F(\la(\rho),\mu(\rho),\rho)\\
 &= 2\(\mu(\rho)\)^2\alpha'(\rho) + 4\mu(\rho)\la(\rho)\beta'(\rho)
  = 2\(\mu(\rho)\)^2\big(\alpha'(\rho) + 2\Lambda(\rho)\beta'(\rho)\big).
\end{align*}
Thus setting
\begin{equation}\label{8}
\chi(\rho) := \frac{f'(\rho)}{2\(\mu(\rho)\)^2}
 =\alpha'(\rho) + 2\Lambda(\rho)\beta'(\rho)
\end{equation}
we need to show that $\chi(\frac12) < 0$.

For $N=3$ this can be checked explicitely.
%we have
%$$
%\chi\Big(\frac12\Big)
% = \frac{16}{9} + 2 - \frac{16}{25}
%   - 4\Big(\sqrt{1 + \frac43 + \frac{16}{5}} - 1\Big)
% \leq \frac{16}{9} + 2 - \frac{16}{25} - 4 < 0.
%$$
Next observe that
\begin{equation}\label{first}
\alpha\Big(\frac12\Big) < \beta^2\Big(\frac12\Big)
 \quad \hbox{ if } N \geq 4
\end{equation}
because
$$
\begin{aligned}
\alpha\Big(\frac12\Big)
 &= \Big(\frac43\Big)^{N-2} - 1 + \Big(\frac45\Big)^{N-2}
   \leq \Big(\frac43\Big)^{N-2} \leq \frac14 2^{N-1}\\
 &\leq \frac14\left(2^{N-3}(2^{N-1}-4) + 1\right)
   =\frac14\beta^2\Big(\frac12\Big) \quad \hbox{ if } N\geq 4.
\end{aligned}
$$
Then, by \eqref{first}, using the inequality
\begin{equation}\label{ineq}
\sqrt{1+x}-1\geq \frac25x\quad \hbox{for any}\ x\in(0,1),
\end{equation}
for $N \geq 4$ we get
\begin{equation}\label{sob}
\begin{aligned}
\chi\Big(\frac12\Big)
 &= \alpha'\Big(\frac12\Big) + \beta\Big(\frac12\Big)
  		\bigg(\sqrt{1+\frac{4\alpha(\frac12)}{\beta^2(\frac12)}}-1\bigg)
			\beta'\Big(\frac12\Big)\\
 &\leq \alpha'\Big(\frac12\Big)
  + \frac85\frac{\alpha(\frac12)}{\beta(\frac12)}\beta'\Big(\frac12\Big).
\end{aligned}
\end{equation}
We compute
$$
\frac{1}{N-2}\alpha'\Big(\frac12\Big)
 = \Big(\frac43\Big)^{N-1}+2-\Big(\frac45\Big)^{N-1}
$$
and
$$
\frac{1}{N-2}\beta'\Big(\frac12\Big) = -2^{N-1},
$$
and
$$
\frac{\alpha (\frac12)}{\beta (\frac12)}
 = \frac{(\frac43)^{N-2} - 1 + (\frac45)^{N-2}}{2^{N-2}-1}
 \geq \frac{(\frac43)^{N-2} - 1 + (\frac45)^{N-2}}{2^{N-2}}.
$$
Therefore, by \eqref{sob}, for $N \geq 4$ we obtain
$$
\begin{aligned}
\frac{1}{N-2}\chi\Big(\frac12\Big)
 &\leq \frac{1}{N-2}\alpha'\Big(\frac12\Big)
   + \frac{1}{N-2}\frac85\frac{\alpha(\frac12)}{\beta(\frac12)}\beta'(\frac12)\\
 &\leq \Big(\frac43\Big)^{N-1} + 2 - \Big(\frac45\Big)^{N-1}
   - \frac85\frac{(\frac43)^{N-2} - 1 + (\frac45)^{N-2}}{2^{N-2}}2^{n-1}\\
 &= -\frac{28}{15}\Big(\frac43\Big)^{N-2} - 4\Big(\frac45\Big)^{N-2} + \frac{26}{5}.
\end{aligned}
$$
We observe that $\frac{28}{15}(\frac43)^{N-2} > \frac{26}{5}$ for $N\geq 6$,
while, by a direct computation,
$-\frac{28}{15}(\frac43)^{N-2}-4(\frac45)^{N-2}+\frac{26}{5} < 0$ for $N=4,5$.
This allows us to conclude that $\chi(\frac12) < 0$ if $N\geq 4$, hence
\eqref{one} holds and Lemma~\ref{exis} follows.
\end{altproof}

For the profile of the second solution we also need the estimate
\begin{equation}\label{two}
f'(\br) < 0 \quad \text{for } \br := \frac{\sqrt{5}-1}{2},
\text{ for all } N\geq 3.
\end{equation}
It is sufficient to prove the inequality $\chi(\br) < 0$, which can be checked
for $N = 3,4,5$ by a direct computation. For the case $N \ge 6$ we first
observe that
\begin{equation}\label{an0}
\sqrt{1+t}-1\ge \frac{t}{3}\quad\forall\ t\in[0,3],
\end{equation}
and
\begin{equation}\label{an2}
1+\br^2\geq2\br\ge\br\sqrt{1+\br^2}\ge1-\br^2=\br.
\end{equation}
In order to use \eqref{an0} we need to check that
\begin{equation}\label{an3}
\frac{4\alpha(\br)}{\beta^2(\br)} \le 3 \quad\hbox{if } N\ge6.
\end{equation}
In fact, \eqref{an2} implies
\begin{align*}
&4\alpha(\br)-3\beta^2(\br)\\
&\hspace{1cm}
 = 4\(\frac{1}{(1 - \br^2)^{N-2}} - \frac{1}{(2\br)^{N-2}}
      +\frac{1}{(1 + \br^2)^{N-2}}\)
	 - 3\(\frac{1}{\br^{2N-4}} + 1 - 2\frac{1}{\br^{N-2}}\)\\
&\hspace{1cm}
 \le \frac{10}{\br^{N-2}} - \frac{3}{\br^{2N-4}}
\end{align*}
and so \eqref{an3} follows because $\br < \(\frac{3}{10}\)^{N-2}$ for
$N \ge 6$.

Therefore, by the definition of $\Lambda$ in \eqref{6}, using \eqref{an0}
and \eqref{an3} we deduce
\begin{equation}\label{an1}
\Lambda(\br) \ge \frac{2\alpha(\br)}{3\beta(\br)}
 > \frac23\alpha(\br)\br^{N-2} \quad\hbox{for } N \ge 6.
\end{equation}
Now, \eqref{an2} and \eqref{an1} combined with the definition of $\chi$ in
\eqref{8} imply for $N \geq 6$:
\begin{align*}
&\chi(\br) \le \alpha'(\br)- \frac43(N-2)\frac{\alpha(\br)}{\br}
\\ &\hspace{1cm}
 = \frac{2(N-2)}{3\br}\(\frac{5\br^2 - 2}{(1 - \br^2)^{N-1}}
   +\frac{7\br}{(2\br)^{N-1}} - \frac{5\br^2 + 2}{(1 + \br^2)^{N-1}}\)
\\ &\hspace{1cm}
 = \frac{2(N-2)}{3\br^n}\(5\br^2 - 2 + \frac{7\br}{2 ^{N-1}}
    - (5\br^2 + 2)\(\frac{\br}{1+\br^2}\)^{N-1}\) < 0.
		\end{align*}
The last inequality follows for $N = 6$ by an explicit computation. For
$N \ge 7$ observe that the term $5\br^2-2 + \frac{7\br}{2 ^{n-1}}$
decreases as $N$ increases, and $5\br^2-2 + \frac{7\br}{2 ^{n-1}} < 0$
for $N = 7$.

This finishes the proof of $\chi(\br) < 0$ for $N \geq 6$, hence \eqref{two}
holds. As a consequence we obtain $\rho_2 > \br$.

\section{The profile of the solutions}\label{profile}

In this section we prove Theorem~\ref{main}~d). Let ${u_1}_\eps$ and
${u_2}_\eps$ be the solutions generated by the critical points
$(\la(\rho_1),\mu(\rho_1),\rho_1)$ and $(\la(\rho_2),\mu(\rho_2),\rho_2)$
of $F$, respectively, as stated in Theorem \ref{th-ex}. Here $\la(\rho)$ and
$\mu(\rho)$ are from \eqref{7}, and $\rho_1,\rho_2$ are the two critical points
of the reduced function $f$ from \eqref{def:f}. Recall that
$0 < \rho_0 < \rho_1 < \frac12$ and $\br = \frac{\sqrt{5}-1}{2} < \rho_2 < 1$.

As a consequence of \eqref{profi} we deduce that in a neighborhood
$\mathfrak{U}$ of $\pa\cB$ not containing the blow-up points there holds
$$
\frac{{u _i }_{\epsilon }(x)}{\alpha_N{\mu_i}\sqrt\eps} \to \varphi(\rho_i,x)
  := \Lambda(\rho _i) G(x,0)- \big(G(x,(\rho_i,0)) + G(x,(-\rho_i ,0))\big)\
 \hbox{ in }  C^1(\mathfrak{U})
$$
as $\eps$ goes to zero. Here $\Lambda $ has been defined in \eqref{7} and
\eqref{6}.

It follows that the exterior normal derivative satisfies
$$
\pa_\nu\varphi(\rho_i,x) = (N-2)\psi(\rho_i,x_1),\quad\ x_1\in[-1,1],
$$
where
\begin{equation}\label{10}
\psi(\rho,x_1)
 := -\Lambda(\rho) + (1-\rho^2)\(\frac{1}{(\rho^2+1-2\rho x_1)^{N/2}}
        + \frac{1}{(\rho^2+1+2\rho x_1)^{N/2}}\).
\end{equation}
Defining
$$
M(\rho) := \max_{|x_1|\le1}\psi(\rho,x_1) = \psi(\rho,1)
   = -\Lambda(\rho) + (1-\rho^2)\(\frac{1}{(1-\rho)^{N}}+\frac{1}{(1+\rho)^{N}}\),
$$
and
$$
m(\rho):=\min_{|x_1|\le1}\psi(\rho,x_1) = \psi(\rho,0)
   = -\Lambda(\rho) + 2(1-\rho^2)\frac{1}{(\rho^2+1)^{N/2}}
$$
we immediately obtain:
$$
m(\rho_i )>0\quad \Longrightarrow\quad
 \pa_\nu u_{i,\eps} \text{ does not change sign in } \pa\Om,
$$
and
$$
m(\rho_i )<0<M(\rho_i )\quad \Longrightarrow\quad
 \pa_\nu u_{i,\eps} \text{ does change sign in } \pa\Om.
$$
Thus Theorem~\ref{main}~d) follows if we can show that
\begin{equation}\label{13}
m(\rho_1) > 0,
\end{equation}
and
\begin{equation}\label{13bis}
m(\rho_2) < 0 < M(\rho_2).
\end{equation}
For the proof of these inequalities we first observe that
\begin{equation}\label{14}
\Lambda' = \frac{\beta'\(\beta-\sqrt{\beta^2+4\alpha}\)
             + 2\alpha'}{2\sqrt{\beta^2+4\alpha}} > 0
\quad \text{for all } \rho \in (\rho_0,1),
\end{equation}
since $\beta' < 0$, $\alpha' > 0$ and $\alpha > 0$ in $(\rho_0,1)$.
Moreover, using \eqref{14}, a simple calculation shows that
\begin{align}\label{pa3}
& m'(\rho)<0\quad \text{for all } \rho \in (\rho_0,1).
\end{align}

\begin{altproof}{\eqref{13}}
By \eqref{pa3} it suffices to prove that
\begin{equation}\label{mum}
m\Big(\frac12\Big)>0.
\end{equation}
This can be checked for $N=3$ by explicit computation. For $N\ge4$ we
argue as follows. Using the inequality $\sqrt{1+x}-1 < \frac{x}{2}$ for
$x>0$ we obtain
$$
m\Big(\frac12\Big)
 > -\frac{\alpha\(\frac12\)}{\beta(\frac12)} + \frac32\Big(\frac45\Big)^{N/2}
 = -\frac{\(\frac43\)^{N-2} - 1 + \(\frac45\)^{N-2}}{2^{N-2}-1}
   + \frac32\Big(\frac45\Big)^{N/2}
$$
Therefore we only need to show that
$$
\frac38\cdot\(\frac{4}{\sqrt{5}}\)^N >
 \frac32\(\frac{2}{\sqrt{5}}\)^N + \(\frac43\)^{N-2} -1 + \(\frac45\)^{N-2}.
$$
This is easily checked for $N=4$, and then it holds for all $N \ge 4$.
\end{altproof}

\begin{altproof}{\eqref{13bis}}
First of all, we remark that
$M(\rho) > 0$ for any $\rho$ such that $\chi(\rho) = 0$
where $\chi$ has been defined in \eqref{8}. In fact, $\chi(\rho)=0$ implies
$\Lambda(\rho)=-\frac{\alpha'(\rho)}{2\beta'(\rho)}$ and so
\begin{equation}\label{emme}
M(\rho)
 = \frac{\rho^N}{(1+\rho^2)^{N-1}}
   + \frac{2^{N-1}\((1+\rho)^N+(1-\rho)^N-\rho^N\)-(1-\rho^2)^{N-1}}
	        {2^{N-1}(1-\rho^2)^{N-1}}
	>0,
\end{equation}
because a direct calculation shows that
$$
2^{N-1}\((1+\rho)^N + (1-\rho)^N - \rho^N\) - (1-\rho^2)^{N-1}>0
\quad\hbox{for any}\ \rho\in[0,1].
$$

As a consequence we obtain $M(\rho_i) > 0$ because $\chi(\rho_i)=0$. Now
it remains to show that $m(\rho_2) < 0$. Since $m$ is decreasing in
$\rho$ and $\rho_2 > \br$ it suffices to prove $m(\br) < 0$. Using
\eqref{an2}  and \eqref{an1} we obtain
\begin{align*}
&m(\br)
 < -\frac23\(\frac{\br}{1-\br^2}\)^{N-2} + \frac23\frac{1}{2^{N-2}}
   - \frac23\(\frac{\br}{1+\br^2}\)^{N-2}
	 + 2\frac{1-\br^2}{\(\sqrt{1+\br^2}\)^N}\\
&\le \frac23\(\frac{\br}{1-\br^2}\)^{N}
     \[-\frac{(1-\br^2)^2}{\br^2} + 4\(\frac{1-\br^2}{2\br }\)^{N}
		   +3(1-\br^2)\(\frac{1-\br^2}{\br\sqrt{1+\br^2}}\)^N\].
\end{align*}
The last expression is negative for $N=6$, hence for all $N \ge 6$. For
$N=3,4,5$ we show $m(\br) < 0$ by an explicit computation.
\end{altproof}

   \end{document}